\documentclass{amsart}
\usepackage[margin=3.5cm]{geometry}
\usepackage{xcolor}

\usepackage{amsfonts, amsthm, amssymb,verbatim,amscd,amsmath, blindtext}

\usepackage{multirow}
\usepackage{graphicx}
\usepackage{enumitem,yhmath}
\usepackage{amssymb}
\usepackage{ytableau}
\usepackage{tikz}
\usepackage{tikz-cd}
\usepackage{float, pifont}
\usepackage{mathtools}
\usepackage[pagebackref,colorlinks=true,citecolor=blue,linkcolor=black]{hyperref}

\newtheorem{theorem}{Theorem}[section]
\newtheorem{lemma}[theorem]{Lemma}
\newtheorem{corollary}[theorem]{Corollary}
\newtheorem{proposition}[theorem]{Proposition}
\newtheorem{claim}[theorem]{Claim}

\theoremstyle{remark}

\theoremstyle{definition}
\newtheorem{problem}[theorem]{Problem}

\makeatletter

\def\env@sqcases{%
  \let\@ifnextchar\new@ifnextchar
  \left\lbrack
  \def\arraystretch{1.2}%
  \array{@{}l@{\quad}l@{}}%
}
\makeatother

\DeclareMathOperator{\reg}{reg}
\DeclareMathOperator{\lcm}{lcm}
\DeclareMathOperator{\pd}{pd}
\DeclareMathOperator{\supp}{supp}

\DeclareMathOperator{\mingens}{Mingens}

\newcommand{\Kfoursymb}[1][1]{%
    \begin{tikzpicture}[scale=#1, %
    baseline=-0.3ex,
    thick,
    ]
    \def\x{2.5mm};
    \def\r{0.2mm};
    \draw (0,0) -- (\x, 0) -- (\x, \x) -- (0, \x) -- cycle;
    \draw (0, 0) -- (\x, \x);
    \draw (\x, 0) -- (0, \x);

    \draw[fill=black] (0,0) circle (\r);
    \draw[fill=black] (0,\x) circle (\r);
    \draw[fill=black] (\x,0) circle (\r);
    \draw[fill=black] (\x,\x) circle (\r);
    \end{tikzpicture}%
}

\newcommand{\diamondsymb}[1][1]{%
    \begin{tikzpicture}[scale=#1, %
    baseline=-0.3ex,
    thick,
    ]
    \def\x{2.5mm};
    \def\r{0.11mm};
    \draw (0,0) -- (\x, 0) -- (\x, \x) -- (0, \x) -- cycle;
    \draw (0, 0) -- (\x, \x);

    \draw[fill=black] (0,0) circle (\r);
    \draw[fill=black] (0,\x) circle (\r);
    \draw[fill=black] (\x,0) circle (\r);
    \draw[fill=black] (\x,\x) circle (\r);
    \end{tikzpicture}%
}

\newcommand{\pawsymb}[1][1]{%
    \begin{tikzpicture}[scale=#1, %
    thick,
    baseline=-1ex
    ]
    \def\x{2.5mm};
    \def\r{0.11mm};
    \draw (\x/2, \x/2) -- (\x, 0) -- (\x/2, -\x/2) -- cycle;
    \draw (\x,0) -- (3*\x/2,0);

    \draw[fill=black] (\x/2, \x/2) circle (\r);
    \draw[fill=black] (\x, 0) circle (\r);
    \draw[fill=black] (\x/2, -\x/2) circle (\r);
    \draw[fill=black] (3*\x/2,0) circle (\r);
    \end{tikzpicture}%
}

\newcommand{\clawsymb}[1][1]{%
    \begin{tikzpicture}[scale=#1, %
    thick,
    baseline=-1ex
    ]
    \def\x{2.5mm};
    \def\r{0.11mm};
    \draw (\x/2, \x/2) -- (\x, 0) -- (\x/2, -\x/2) ;
    \draw (\x,0) -- (3*\x/2,0);

    \draw[fill=black] (\x/2, \x/2) circle (\r);
    \draw[fill=black] (\x, 0) circle (\r);
    \draw[fill=black] (\x/2, -\x/2) circle (\r);
    \draw[fill=black] (3*\x/2,0) circle (\r);
    \end{tikzpicture}%
}

\def\x{{\bf x}}

\def\1{{\bf 1}}
\def\0{{\bf 0}}
\def\r{{\mathbf r}^{\mathbb L}}

\begin{document}
\title[Scarfness of some ideals associated to graphs]{\textbf{The Scarf complex of squarefree powers, symbolic powers of edge ideals, and cover ideals of graphs}}
\author{Trung Chau}
\address{Chennai Mathematical Institute, Siruseri, Tamil Nadu, India}
\email{chauchitrung1996@gmail.com}

\author{Nursel Erey}
\address{Gebze Technical University, Department of Mathematics, Gebze, 41400 Kocaeli,
Turkey}
\email{nurselerey@gtu.edu.tr}

\author{Aryaman Maithani}
\address{Department of Mathematics, University of Utah, 155 South 1400 East, Salt Lake City, UT~84112, USA}
\email{maithani@math.utah.edu}



\begin{abstract}
     Every monomial ideal $I$ has a Scarf complex, which is a subcomplex of its minimal free resolution. We say that $I$ is Scarf if its Scarf complex is also its minimal free resolution. In this paper, we fully characterize all pairs $(G,n)$ of a graph $G$ and an integer $n$ such that the squarefree power $I(G)^{[n]}$ or the symbolic power $I(G)^{(n)}$ of the edge ideal $I(G)$ is Scarf. We also determine all graphs $G$ such that its cover ideal $J(G)$ is Scarf, with an explicit description when $G$ is either chordal or bipartite.
\end{abstract}

\maketitle

\section{Introduction}

The study of homological invariants, e.g., Betti numbers and (Castelnuovo-Mumford) regularity,  of monomial ideals have played a central role in Commutative Algebra over the past few decades due to its natural connection to combinatorial objects, particularly graphs. Let $G=(V(G),E(G))$ be a finite simple graph. The \emph{edge ideal} $I(G)$ of $G$ is an ideal of the polynomial ring $S=\Bbbk[V(G)]$ and is generated by the monomials $xy$ where $\{x,y\}\in E(G)$. The \emph{cover ideal} of $G$, denoted by $J(G)$, is defined as the ideal generated by the monomials $\prod_{x\in W} x$ where $W$ is a vertex cover of $G$.
Edge ideals and cover ideals have been studied intensively due to their simple constructions and direct connection to Graph Theory and Combinatorics. We refer to \cite{MV-survey} and \cite{VanTuyl} for some surveys on the subject.

Homological invariants of monomial ideals can be read off their minimal free resolutions. Unfortunately, the study of minimal free resolutions, or one can say the study of syzygies, is notoriously difficult in general. In the case of monomial ideals, there are many general constructions for free resolutions (cf. \cite{AFG2020, BW02, CK24, CT2016, Ly88, Nov00, OY2015, Tay66}), albeit they are not always minimal. Bayer, Peeva, and Sturmfels \cite{BPS98} constructed \emph{Scarf complex} of a given monomial ideal, inspired by the work of Herbert Scarf in mathematical economics \cite{Sca73}. It is known that for a monomial ideal $I$, its Scarf complex is a subcomplex of its minimal free resolution and thus serves as a lower bound for it. The ideal $I$ is called \emph{Scarf} if this lower bound is achieved. Bayer, Peeva, and Sturmfels in \cite{BPS98} proved that generic monomial ideals are Scarf, and most monomial ideals are generic. However, some of the most interesting monomial ideals, e.g., squarefree monomial ideals, or their powers, are rarely generic. Thus, a classification of when these ideals are Scarf remains a mystery. Scarf ideals have been studied from different perspectives since \cite{BS98,Yu99} and made a recent comeback with some articles on extremal ideals \cite{extremal3, ELKHOURY2024107577} which are conjectured to be Scarf.

Regarding edge ideals and their powers, a full characterization of Scarfness of $I(G)^n$ was recently given in \cite{FHHM24} for any connected graph $G$ and any positive integer $n$. The goal of this paper is to expand the investigation of Scarfness to symbolic powers and squarefree powers of edge ideals as well as cover ideals of graphs. We remark that in this article, we will only consider graphs with no isolated vertices, as isolated vertices do not contribute to edge ideals or cover ideals. Before stating our main results, we refer to Section~\ref{sec:preliminaries} for unexplained terminology.   

In Section~\ref{sec:squarefree powers}, we study the Scarf property of squarefree powers of edge ideals and give a complete characterization of such a property:

\begin{theorem}[{Theorem \ref{thm:Scarf-squarefree-powers}}]
Let $G$ be a graph, and let $2\leq n\leq m(G)$ where $m(G)$ denotes the matching number of $G$. Then $I(G)^{[n]}$ is Scarf if and only if one of the following holds:
    \begin{enumerate}
        \item $G$ has a perfect matching of size $n$.
        \item The vertex set of $G$ can be labelled as $V(G)=\{x_1,\dots, x_{2n},y_1,\dots y_\ell\}$ for some $\ell\geq 1$ such that \[\{x_1x_2, x_3x_4, \dots, x_{2n-3}x_{2n-2}, x_{2n-1}x_{2n}\}\] is a matching of $G$, and $x_{2n},y_1,\dots ,y_\ell$ are leaves of $G$ with the common joint $x_{2n-1}$.
        \end{enumerate}
\end{theorem}

Section~\ref{sec:symbolic powers} is devoted to symbolic powers of edge ideals. In Theorem~\ref{thm:HHZ-Scarf-symbolicpowers} we show that when $H$ is an induced subgraph of $G$, the symbolic powers of $I(H)$ can be obtained from those of $I(G)$ by restriction. With this new technique, we characterize symbolic powers of edge ideals that are Scarf:

\begin{theorem}[{Theorem \ref{thm:Scarf-symbolic-powers}}]
    Let $n\geq 2$ be an integer and let $G$ be a graph. Then $I(G)^{(n)}$ is Scarf if and only if one of the following holds:
    \begin{enumerate}
        \item $G$ is a path of length 1 or 2.
        \item $G$ is a triangle and $n$ is even.
        \item $G$ is a disjoint union of two edges.
    \end{enumerate}
\end{theorem}

In Section~\ref{sec:cover ideals}, we solve the Scarfness problem for cover ideals of graphs. Surprisingly, in contrast to edge ideals, we find an infinite family of graphs whose cover ideals are Scarf. Our main theorem of the last section is stated as follows:

\begin{theorem}[{Theorem~\ref{thm:Scarf-cover-ideals}}]
    Let $G$ be a graph, and let $n$ denote its number of minimal vertex covers. Then the following statements are equivalent.
    \begin{enumerate}
        \item The cover ideal $J(G)$ is Scarf.
        \item $G$ is co-chordal and there are at least $n-1$ pairs of minimal vertex covers $(V_1,V_2)$ such that  the only minimal vertex covers of $G$ within $V_1\cup V_2$ are $V_1$ and $V_2$ themselves.
    \end{enumerate}
\end{theorem}

Moreover, limiting to important classes of graphs such as chordal and bipartite graphs, we can explicitly describe those with Scarf cover ideals.

\begin{theorem}[{Theorem \ref{thm:Scarf-cover-ideal-chordal}}]
    Let $G$ be a chordal graph. Then $J(G)$ is Scarf if and only if $V(G)$ can be partitioned into a clique $A$ and an independent set $B$ such that
    \begin{enumerate}
        \item any vertex in $A$ is adjacent to a vertex in $B$;
        \item no two sets in the collection $\{ N_G(x) \cap B \colon x\in A \}$ contains one another.
    \end{enumerate}
\end{theorem}

\begin{theorem}[{Theorem \ref{thm:Scarf-cover-ideals-bipartite}}]
    Let $G$ be a bipartite graph. Then $J(G)$ is Scarf if and only if $G$ is a Ferrers graph.
\end{theorem}

\section*{Acknowledgements} 

The first author acknowledges support from the Infosys Foundation. The third author was supported by NSF grants DMS~2101671 and DMS~2349623. We would like to thank Professors Huy T\`ai H\`a and Takayuki Hibi for the many discussions without which this work would not be possible.

\section{Preliminaries}\label{sec:preliminaries}

Throughout the paper, $S$ denotes the polynomial ring $\Bbbk[x_1,\dots, x_N]$ over a field $\Bbbk$, and $\mathfrak{m}$ its homogeneous maximal ideal.

\subsection{Simplicial resolutions and complexes}

For any homogeneous ideal $I$ of $S$, a \emph{free resolution} of $S/I$ over $S$ is a complex of $S$-modules 
\[
\mathcal{F}\colon 0\to F_p\to \cdots \to F_1\to F_0\to 0
\]
where $H_0(\mathcal{F})\cong S/I$ and $H_i(\mathcal{F})=0$ for any $i>0$. Moreover, $\mathcal{F}$ is called \emph{$\mathbb{N}^d$-graded} for some integer $d$ if each differential is $\mathbb{N}^d$-homogeneous, and \emph{minimal}  if $\partial(F_i) \subseteq \mathfrak{m}F_{i-1}$ for each $i> 0$. 

Let $I=(m_1,\dots, m_q)$ be a monomial ideal in a polynomial ring $S$. We denote by $\mingens(I)$ the minimal set of monomial generators of $I$. Set $a_m$ to be the exponent vector of a monomial $m$. It is well-known (cf. \cite{Tay66}) that the full $q$-simplex with vertices labeled with monomial generators of $I$ induces a free resolution of $S/I$, called the \emph{Taylor's resolution}:
\[
\mathcal{T}_I\colon 0\to T_q \xrightarrow{\partial} \cdots \xrightarrow{\partial} T_1\xrightarrow{\partial} T_0\to 0
\]
where
\[
T_i\coloneqq \bigoplus_{\substack{\sigma \subseteq \mingens(I)\\
|\sigma|=i}} Se_{\sigma}
\]
for each $i\geq 0$, and
\[
\partial(e_{\sigma}) = \sum_{j=1}^t (-1)^{j+1} \frac{\lcm(\sigma)}{\lcm(\sigma \setminus \{m_{i_j}\})} e_{\sigma \setminus \{m_{i_j}\}}
\]
where $\sigma= \{m_{i_1},\dots, m_{i_t}\}$. We note that $Se_{\sigma}= S(-a_{\lcm(\sigma)})$. The shifts here are needed to make $\mathcal{T}_I$ $\mathbb{N}^N$-graded. We  label each face of this full $q$-simplex, which corresponds to a subset $\sigma\subseteq \mingens(I)$, with the monomial $\lcm(\sigma)$. 

A \emph{simplicial complex} $\Delta$ with $q$ vertices labeled with monomial generators of $I$ naturally induces a complex of $S$-modules, denoted by $\mathcal{T}_\Delta$, that is canonically a subcomplex of the Taylor resolution $\mathcal{T}_I$ of $S/I$. Recall that a simplicial complex is called \emph{acyclic} if it is connected and all its positive-order homology groups are zero. Bayer, Peeva, and Sturmfels gave a characterization on when this complex is a resolution of $S/I$. 

\begin{lemma}[{\cite[Lemma~2.2]{BPS98}}]\label{lem:BPS}
    Let $I$ be a monomial ideal and let $\Delta$ be a simplicial complex with vertices labeled with elements in $\mingens(I)$. Then the complex $\mathcal{T}_\Delta$ is a resolution of $S/I$ if and only if for any monomial $m$, the simplicial complex $\Delta_{\leq m} = \{ \sigma  \in \Delta \colon \lcm(\sigma) \mid m \} $ is either acyclic or empty over $\Bbbk$.
\end{lemma}

A particular simplicial complex we are interested in is the \emph{Scarf complex}, formed by the faces of the $q$-simplex with unique labels. It is known that this is a simplicial complex (see, e.g., \cite[Remark~5.2]{Mer09}). We will refer to a face of the Scarf complex as a Scarf face. The corresponding complex of $S$-modules is called the \emph{Scarf complex} of $S/I$. In general, the Scarf complex is not a resolution, but is always a subcomplex of the minimal resolution of $S/I$ (cf. \cite[Theorem~59.2]{Peeva10}). Thus, if it is a resolution, it is the minimal resolution.

We call the monomial ideal $I$ Taylor (resp. Scarf) if the Taylor resolution (resp. Scarf complex) of $S/I$ is its minimal free resolution. The following result is straightforward, and we include a proof for completeness.

\begin{lemma}\label{lem:Taylor-implies-Scarf}
    Taylor ideals are Scarf.
\end{lemma}
\begin{proof}
    The result is equivalent to saying that if the Taylor resolution of a monomial ideal $I$ is minimal, then each label $\lcm(\sigma)$ where $\sigma\subseteq \mingens(I)$ is unique. We will show this with contraposition. Let $\sigma$ be a subset of $\mingens(I)$ such that its label $\lcm(\sigma)$ is not unique. It suffices to show that the Taylor resolution of $S/I$ is not minimal.

    Indeed, since $\lcm(\sigma)$ is not unique, there exists $\sigma\neq \sigma'\subseteq \mingens(I)$ such that $\lcm(\sigma') = \lcm(\sigma)$. Set $\tau\coloneqq \sigma \cup \sigma'$, and consider any $m\in (\sigma \setminus \sigma') \cup (\sigma'\setminus \sigma)$. Then $m\in \tau$, and $\lcm(\tau) = \lcm(\tau \setminus \{m\})$. Then inspecting the differential of the Taylor resolution of $S/I$, we observe that
    \[
    \partial(e_{\tau}) = \pm e_{\tau\setminus \{m\}} + \cdots.
    \]
    By definition, the Taylor resolution is not minimal, as desired.
\end{proof}

Lastly, we note the well-known fact that a pure $1$-dimensional simplicial complex is connected if and only if it is shellable. Therefore, it follows from \cite[Corollary~3.1.4]{Wachs} that a graph, as a $1$-dimensional simplicial complex, is acyclic if and only if it is a tree.

\subsection{Graphs, edge ideals, and cover ideals}

Let $G=(V(G),E(G))$ be a finite simple graph with $V(G)=\{x_1,\dots ,x_N\}$. The \emph{edge ideal} $I(G)$ of $G$ is an ideal of $S=\Bbbk[x_1,\dots ,x_N]$ defined~by
\[
I(G) = (xy\colon \{x,y\} \in E(G)).
\]
To simplify the notation, we will use an edge $\{x,y\}$  interchangeably with the corresponding monomial generator $xy$. A \emph{vertex cover} of $G$ is a subset of vertices of $G$ such that each edge of $G$ contains at least one of these vertices. 
The \emph{cover ideal} $J(G)$ of $G$ is defined to be 
\[
J(G) = \left(\prod_{x\in W} x \colon W\subseteq V(G) \text{ is a vertex cover of } G  \right).
\]
We remark that the minimal generators of the cover ideal $J(G)$ correspond to minimal vertex covers of $G$. It is clear that the edge ideal and cover ideal are unchanged after adding isolated vertices to the graph. Thus we will assume that all graphs in this paper have no isolated vertices.

A \emph{matching} of $G$ is a set of edges of $G$, among which no two edges share a common vertex. The \emph{matching number} of $G$, denoted by $m(G)$, is the maximum cardinality of a matching of $G$.

An \emph{induced subgraph} $H$ of $G$ is a subgraph such that $\{x,y\}\in E(H)$ whenever $\{x,y\}\in E(G)$ and $x,y\in V(H)$.

We denote by $C_n$ the $n$-cycle graph for any $n\geq 3$. In particular, $C_3$ is a triangle. A \emph{tree} is a connected graph that does not contain any $n$-cycle as a subgraph. A \emph{bipartite} graph is a graph that has no odd-cycle subgraph. A \emph{chordal} graph is a graph that does not contain any $n$-cycles as an induced subgraph for any $n\geq 4$. A \emph{co-chordal} graph is a graph whose complement is chordal.

A set $W$ of vertices of $G$ is called an \emph{independent} set if no two vertices in $W$ are adjacent. A \emph{clique} of $G$ is a subgraph of $G$ which is a complete graph. A \emph{leaf} of $G$ is a vertex which has only one neighbour. The unique neighbour of a leaf is called its \emph{joint}.

\subsection{Restriction Lemma and HHZ-subideals}

There is a bijective correspondence between $\mathbb{N}^{N}$ and the set of all monomials in $S$. We will thus use monomials to sometimes denote vectors in $\mathbb{N}^{N}$.

Fix a monomial ideal $I$ and a monomial $m$. Let $I^{\leq m}$ be the ideal generated by all elements of $\mingens(I)$ that divide $m$. It is clear that $I^{\leq m}$ is always a subideal of $I$. This notation was introduced in \cite{HHZ04}, although it is worth mentioning that the idea appeared briefly in \cite{BPS98}. We will call $I^{\leq m}$ a \emph{Herzog-Hibi-Zheng-subideal} (with respect to $m$) of $I$, or \emph{HHZ-subideal} for short. 

Let 
\[
\mathcal{F}: 0\to F_n \to F_{n-1}\to \dots \to F_1\to F_0 \to 0
\]
be a (minimal) $\mathbb{N}^{N}$-graded free resolutions of $S/I$, with $F_i= \bigoplus_j S(-q_{ij})$. Let $\mathcal{G}$ be the subcomplex of $\mathcal{F}$ where
\[
G_i\coloneqq \bigoplus_{q_{ij}\leq m} S(-q_{ij}).
\]
Here by $a\leq b$ for two monomials $a$ and $b$, we mean $a\mid b$. Then by \cite[Lemma~4.4]{HHZ04}, $\mathcal{G}$ is a (minimal) $\mathbb{N}^d$-graded free resolution of $S/I^{\leq m}$. This result is commonly known as the \emph{Restriction Lemma}. It is straightforward that the restriction of the Taylor resolution of $S/I$ with respect to a monomial $m$ is the Taylor resolution of $S/I^{\leq m}$, and that the same holds for Scarf complexes.

\begin{lemma}[{Restriction Lemma for Taylor resolutions and Scarf complexes}]\label{lem:HHZ-Taylor-Scarf}
     Let $I$ be a monomial ideal and let $m$ be a monomial. If $I$ is Taylor (or Scarf), then so is $I^{\leq m}$.
\end{lemma}

We record here an unsurprising result which will be used throughout the paper.

\begin{lemma}\label{lem:multiply-by-a-monomial}
    Let $I$ be a monomial ideal, and $m$ a monomial. Then, $I$ is Scarf if and only if $mI$ is.
\end{lemma}

\begin{proof}
    It is straightforward that a label $\lcm(\sigma)$, where $\sigma\subseteq \mingens(I)$, for the ideal $I$ is unique if and only if the label $m \lcm(\sigma)$ is unique among the possible labels for the ideal $mI$. In particular, this implies that the two Scarf complexes of $S/I$ and $S/mI$ are induced by the same simplicial complex (where vertices labelled with corresponding generators). The result the follows.
\end{proof}

\section{Squarefree powers of edge ideals}\label{sec:squarefree powers}

The \emph{$n$-th squarefree power} of an edge ideal $I(G)$, denoted by $I(G)^{[n]}$, is the ideal generated by the squarefree monomials of $I(G)^n$. In other words, for $m=\prod_{x\in V(G)} x$, we can define \[I(G)^{[n]}= \big( I(G)^n \big)^{\leq m}\]
and therefore $I(G)^{[n]}$ is an HHZ-subideal of $I(G)^n$.
Squarefree powers of edge ideals were introduced in \cite{BHZ} and quickly gained traction \cite{EHHS24, EREY2022105585, EH21, FHH2023, SF2024}.
It is clear that the minimal generators of $I(G)^{[n]}$ correspond to the matchings of size $n$ of $G$. Thus, $I(G)^{[n]} \neq  0$ if and only if $n\leq m(G)$. To omit trivial cases, we assume that $n\leq m(G)$ for the rest of the section.

We first prove some basic rules that will be useful later when we apply restriction to powers of ideals. 

\begin{proposition}\label{prop:restriction and power switch}
Let $I$ be a monomial ideal.

\begin{enumerate}
 \item The equality $I=I^{\leq w}$ holds for any monomial $w$ which is divisible by the least common multiple of the minimal monomial generators of $I$.
    \item If $m$ is a monomial and $n$ is a positive integer, then 
    \[(I^n)^{\leq m}=[(I^{\leq m})^n]^{\leq m}.\]
\end{enumerate}
\end{proposition}

\begin{proof}
(1) is immediate from the definition. To prove (2), let $\mingens({I^{\leq m}})=\{u_1,\dots ,u_r\}=U$ and $\mingens(I)=\{v_1,\dots ,v_s\}=V$. First, suppose $w$ is a monomial in $(I^n)^{\leq m}$. Then $v_{i_1}\dots v_{i_n}$ divides $w$ for some $i_1,\dots ,i_n$ with $v_{i_1}\dots v_{i_n}\mid~m$. Then $\{v_{i_1},\dots ,v_{i_n}\}\subseteq U$ and $w$ is a monomial in $[(I^{\leq m})^n]^{\leq m}$.

Now, suppose $w$ is a monomial in $[(I^{\leq m})^n]^{\leq m}$. Then $u_{j_1}\dots u_{j_n}$ divides $w$ for some $j_1,\dots ,j_n$ with $u_{j_1}\dots u_{j_n}\mid~m$. Since $U\subseteq V$, it is clear that $w$ is a monomial in $(I^n)^{\leq m}$.
\end{proof}

The application of the The Restriction Lemma to squarefree powers of edge ideals of induced subgraphs was given in \cite[Corollary~1.3]{EREY2022105585}. We include the proof of the restriction process below for completeness.

\begin{lemma}\label{lem:HHZ-Scarf-squarefreepower}
    (Restriction Lemma for the Scarf complexes of squarefree powers) Let $H$ be an induced subgraph of $G$, and let $n$ be a positive integer. Let $m$ be the product of all vertices of $H$. Then,
    \[I(H)^{[n]}=( I(G)^{[n]})^{\leq m}.\]
    In particular, if $I(G)^{[n]}$ is Scarf, then so is $I(H)^{[n]}$.
\end{lemma}
\begin{proof} Since $I(G)^{\leq m}=I(H)$ it follows from the definition and Proposition~\ref{prop:restriction and power switch}(2) that
    \[
    I(H)^{[n]} = ( I(H)^{n})^{\leq m} = [ (I(G)^{\leq m})^{n}]^{\leq m}= ( I(G)^{n} )^{\leq m} = ( I(G)^{[n]})^{\leq m}.
    \]
    The last statement of the lemma then follows from Lemma~\ref{lem:HHZ-Taylor-Scarf}.
\end{proof}

The goal of this section is the full characterization of pairs $(G,n)$ such that the squarefree power $I(G)^{[n]}$ is Scarf. Since $I(G)^{[1]}=I(G)$, we already have a full characterization in this case \cite[Theorem~8.3]{FHHM24}. For the rest of the section, we assume that $n\geq 2$. 

Next, we characterize when two monomial generators of $I(G)^{[n]}$ form an edge in its Scarf complex. For any monomial $u$, we denote by $\supp(u)$, either the set of variables dividing $u$ or the product of such variables, depending on the context. 

\begin{lemma}[Scarf edge lemma]\label{lem:Scarf-edge}
    Let $G$ be a graph and let $n\geq 2$. Let $m_1$ and $m_2$ be distinct generators in $\mingens( I(G)^{[n]})$. If $\{m_1,m_2\}$ is Scarf, then $\deg (\lcm(m_1,m_2))=2n+1$.
\end{lemma}

\begin{proof}
    We will prove the contrapositive. Suppose that $\deg (\lcm(m_1,m_2))>2n+1$, i.e.,  
    \[ |\supp(m_2)\setminus \supp(m_1)|\geq 2.\]
    We will show that $\{m_1,m_2\}$ is not  Scarf. Without loss of generality, we assume that
    \[
    m_1=(x_1x_2)\cdots(x_{2n-1}x_{2n})
    \]
    for some edges $x_1x_2,\dots, x_{2n-1}x_{2n}$ of $G$. If there exist two vertices $y,z\in \supp(m_2)\setminus \supp(m_1)$ such that $yz$ is an edge of $G$, then set
    \begin{align*}
        m_3&\coloneqq (x_3x_4)\cdots (x_{2n-1}x_{2n}) (yz) \in \mingens( I(G)^{[n]}),\\
        m_4&\coloneqq (x_1x_2)\cdots (x_{2n-3}x_{2n-2}) (yz) \in \mingens( I(G)^{[n]}).
    \end{align*}
    It is clear that $m_1, m_3$, and $m_4$ are distinct generators. Then, either $m_3\neq m_2$ or $m_4\neq m_2$. Without loss of generality, assume that $m_3\neq m_2$. Since $\lcm(m_1,m_2)=\lcm(m_1,m_2,m_3)$, the edge $\{m_1,m_2\}$ is not Scarf, as desired. Now, we can assume that $\supp(m_2)\setminus \supp(m_1)$ is an independent set of $G$. We consider a vertex $y \in \supp(m_2)\setminus \supp(m_1)$. Since $y$ divides $m_2$, it must form an edge of $G$ with another vertex in $\supp(m_2)$, by our above assumption, without loss of generality, we can assume that $x_1y$ is an edge of $G$. Set
    \[
    m_3\coloneqq (x_3x_4)\cdots (x_{2n-1}x_{2n}) (x_1y) \in \mingens( I(G)^{[n]}). 
    \]
    It is clear that $m_3\neq m_1,m_2$ because $|\supp(m_2)\setminus \supp(m_1)|\geq 2$. Moreover, since $\lcm(m_1,m_2)=\lcm(m_1,m_2,m_3)$, we conclude that the set $\{m_1,m_2\}$ is not Scarf, as desired.
\end{proof}

We characterize all Scarf monomial ideals under some strict assumptions.

 \begin{lemma}\label{lem:Scarf-2-gens}
        Let $t$ be a positive integer and let $I$ be a squarefree monomial ideal of $\Bbbk [x_1,\dots, x_{t+1}]$ generated by monomials of degree $t$. Then $I$ is Scarf if and only if $I$ is minimally generated by at most two monomials.
\end{lemma}

    \begin{proof}
        If $I$ is generated by at most two generators, then its Taylor resolution is minimal, and hence $I$ is Scarf. Conversely, assume that $\mingens(I)=\{m_1,\dots, m_n\}$ where $n\geq 3$ and let $\sigma\subseteq \mingens(I)$ with $|\sigma|\geq 2$. Notice that the monomial $\lcm(\sigma)$ is of degree of at least $t+1$, and since it is squarefree, we must have $\lcm(\sigma) \mid x_1\dots x_{t+1}$. Therefore, $\lcm(\sigma)=x_1\dots x_{t+1}$, and in particular, since $n\geq 3$, the label $x_1 \cdots x_{t+1}$ is not unique. Thus, the Scarf complex of $I$~is
            \[
            0 \to \bigoplus_{i=1}^{t+1} S[\hat x_i] \to S[\varnothing] \to 0
            \]
        where $\hat x_i = \prod_{j\neq i}x_j$. The simplicial complex associated to this chain complex consists of $t+1$ disconnected points. By Lemma~\ref{lem:BPS}, it does not support a resolution of $S/I$. In particular, $I$ is not Scarf. This concludes the proof.
    \end{proof}

The strong assumptions in this lemma, even though strict, apply to the squarefree powers of edge ideals. In particular, this forces certain structures on $I(H)^{[n]}$ where $H$ is any induced subgraph of $G$ with $2n+1$ vertices. The following lemma will serve as a basis for the main result (Theorem~\ref{thm:Scarf-squarefree-powers}) of this section.

\begin{lemma}\label{lem:2n+1-vertices}
    Let $G$ be a graph on $2n+1$ vertices without isolated vertices. Let $2\leq n\leq m(G)$. Then the following statements are equivalent.
    \begin{enumerate}
        \item $I(G)^{[n]}$ is Scarf.
        \item $\left|\mingens(I(G)^{[n]})\right|=2$.
        \item The vertices of $G$ can be labeled with $x_1,\dots, x_{2n-1},y,z$ such that $G$ contains the edges $x_1x_2,\; x_3x_4,\; \dots,\; x_{2n-3}x_{2n-2},\; x_{2n-1}y,\; x_{2n-1}z$, and $y$, $z$ are leaves of $G$. 
        \item $\mingens(I(G)^{[n]})=\{m_1, m_2\}$ such that $\supp(m_1)\setminus\supp(m_2)=\{y\}$ and $\supp(m_2)\setminus\supp(m_1)=\{z\}$ for some leaves $y$ and $z$ of $G$ that are adjacent to the same joint.
    \end{enumerate}  
\end{lemma}

\begin{proof}
    It is straightforward to show that (3)$\iff$(4). The equivalence of (1) and (2) follows from Lemma~\ref{lem:Scarf-2-gens}. If (3) holds, then $I(G)^{[n]} = (\prod_{i=1}^{2n-1}x_i)(y,z)$ which implies (2).

    We only need to show that (1)$\implies$(3). Assume that $I(G)^{[n]}=(m_1,m_2)$ is Scarf. By Lemma~\ref{lem:Scarf-edge}, we may assume that $\supp(m_1)\setminus \supp(m_2)=\{y\}$ and $\supp(m_2)\setminus \supp(m_1)=\{z\}$ for some vertices $y$ and $z$ since the generators $m_1$ and $m_2$ have $2n-1$ vertices in common. Suppose that 
    \[m_1=(x_1x_2)\cdots(x_{2n-3}x_{2n-2})(x_{2n-1}y)\]
    where $x_1x_2,\dots, x_{2n-3}x_{2n-2}, x_{2n-1}y$ are edges of $G$. We claim that $x_{2n-1}$ is the only vertex of $G$ that is adjacent to $z$. If $yz$ is an edge of $G$, then 
    \[
    m_1,m_2\neq (x_1x_2)\cdots (x_{2n-3}x_{2n-2})(yz) \in I(G)^{[n]},
    \]
    a contradiction to (3). Without loss of generality, suppose that $x_1z$ is an edge of $G$. Then
    \[
    m_1,m_2\neq (x_1z)(x_3x_4)\cdots (x_{2n-1}y) \in I(G)^{[n]},
    \]
    a contradiction to (3). Thus $z$ is a leaf of $G$ and it is adjacent to $x_{2n-1}$. By the symmetry, it follows that $y$ is also a leaf of $G$.
\end{proof}

We recall that a \emph{perfect matching} of a graph is a matching that involves all of its vertices. It turns out that another restriction for when $I(G)^{[n]}$ is Scarf is the matching number $m(G)$ of $G$. We show that among the nonvanishing squarefree powers, only the highest one can be Scarf. 

\begin{lemma}\label{lem:matching-number-Erey}
    Let $2\leq n\leq m(G)$. If $I(G)^{[n]}$ is Scarf, then $m(G)=n$.
\end{lemma}

\begin{proof}
    We claim that if $G$ has a matching of size $n+1$, then $I(G)^{[n]}$ is not Scarf. By restricting to the vertices in this matching and using Proposition~\ref{lem:HHZ-Scarf-squarefreepower}, we may assume for a contradiction that $G$ has a perfect matching of size $n+1$ but $I(G)^{[n]}$ is Scarf. Let $x_1,\dots ,x_{2n+2}$ be the distinct vertices of $G$ such that $x_1x_2, x_3x_4,\dots , x_{2n+1}x_{2n+2}$ form the edges of the perfect matching of $G$. 
    
    For each $1\leq i\leq n+1$, set  
    \[
    m_i= \frac{\prod_{j=1}^{2n+2}x_j}{x_{2i-1}x_{2i}}= (x_1x_2) \cdots \widehat{(x_{2i-1}x_{2i})} \cdots (x_{2n+1}x_{2n+2}) \in I(G)^{[n]}.
    \]
    If these are the only generators of $I(G)^{[n]}$, then its Scarf complex does not contain any edge due to Lemma~\ref{lem:Scarf-edge}, and thus is not acyclic over $\Bbbk$. By Lemma~\ref{lem:BPS}, the ideal $I(G)^{[n]}$ is not Scarf, a contradiction. Thus we can assume that there exists a generator $\mathbf{m}\in I(G)^{[n]}$ such that $\{m_{n+1},\mathbf{m}\}$ is a Scarf edge. By Lemma~\ref{lem:Scarf-edge}, $m_{n+1} $ and $\mathbf{m}$ share $2n-1$ vertices, or in other words, they contain $2n+1$ vertices in total. Without loss of generality, assume that these $2n-1$ vertices are $x_1,\dots, x_{2n-1}$, and that $x_{2n+1}$ divides $\mathbf{m}$. Let $H$  be the induced subgraph of $G$ on $V(G)\setminus \{x_{2n+2}\}$. By Lemma~\ref{lem:HHZ-Scarf-squarefreepower}, the ideal $I(H)^{[n]}$ is Scarf. Lemma~\ref{lem:2n+1-vertices} implies that
    \[
    \mathbf{m}=(x_1x_2)\cdots (x_{2n-3}x_{2n-2})(x_{2n-1}x_{2n+1}),
    \]
    where $x_{2n}$ and $x_{2n+1}$ are leaves of $H$ with the common joint $x_{2n-1}$. Using the same argument for the induced subgraph of $G$ on $V(G)\setminus \{x_{2n}\}$, we deduce that $x_{2n-1}$ and $x_{2n+2}$ do not form an edge with any of the vertices $x_1,\dots, x_{2n-2}$ in $G$. In particular, this means that
    \begin{equation}\label{0,2,4}
        |\{x_{2n-1},x_{2n},x_{2n+1},x_{2n+2}\}\cap \supp(\mathbf{u})| \in \{2,4\}
    \end{equation}
    for any minimal generator $\mathbf{u}$ of $I(G)^{[n]}$. Now, consider the vertices $m_1$ and $m_{n+1}$ in the Scarf complex of $I(G)^{[n]}$. By Lemma~\ref{lem:BPS}, this simplicial complex, over $\Bbbk$, is acyclic, and thus there exists a path in this complex that connects the two vertices $m_1$ and $m_{n+1}$, i.e., there exist Scarf edges 
    \[\{m_1,\mathbf{m}_1\}, \{\mathbf{m}_1,\mathbf{m}_2\}, \dots, \{\mathbf{m}_\ell,m_{n+1}\}\] 
    for some integer $\ell$, and some generators $\mathbf{m}_1, \dots, \mathbf{m}_\ell \in I(G)^{[n]}$. Since $n\geq 2$ we have
    \[
    |\{x_{2n-1},x_{2n},x_{2n+1},x_{2n+2}\}\cap \supp(m_1)| = 4.
    \]
    By Lemma~\ref{lem:Scarf-edge}, the generators $\mathbf{m}_1$ and $m_1$ share $2n-1$ vertices, and thus by ($\ref{0,2,4}$), we have
    \[
    |\{x_{2n-1},x_{2n},x_{2n+1},x_{2n+2}\}\cap \supp(\mathbf{m}_1)| = 4.
    \]
    Repeating this argument, we obtain
    \[
    |\{x_{2n-1},x_{2n},x_{2n+1},x_{2n+2}\}\cap \supp(\mathbf{m}_i)| = 4
    \]
    for any $1\leq i\leq \ell$, and also
    \[
    2=|\{x_{2n-1},x_{2n},x_{2n+1},x_{2n+2}\}\cap \supp(m_{n+1})| = 4,
    \]
    which is a contradiction, as desired.
\end{proof}

Next, we will construct a graph $G$ with $2n+2$ vertices such that $I(G)^{[n]}$ is not Scarf. This particular graph will serve as a forbidden structure in the proof of our main theorem.

\begin{lemma}\label{lem:forbidden-squarefree}
    Let $n\geq 2$ and let $G$ be a graph on $2n+2$ vertices $V(G)=\{x_1,\dots, x_{2n-2},y_1,y_2, y_3, y_4\}$ such that $G$ contains the edges 
    \[
    x_1x_2, x_3x_4, \dots, x_{2n-5}x_{2n-4}, x_{2n-3}y_1, x_{2n-3} y_2,  x_{2n-2}y_3, x_{2n-2}y_4.
    \]
    Then $I(G)^{[n]}$ is not Scarf.
\end{lemma}

\begin{proof}
    Assume for a contradiction $I(G)^{[n]}$ is Scarf. Then $m(G)=n$ by Lemma~\ref{lem:matching-number-Erey}. We claim that $y_1,y_2,y_3$ and $y_4$ are leaves of $G$. Let $H_1$ be the subgraph of $G$ on $V(G)\setminus \{y_4\}$. Then Lemmas~\ref{lem:HHZ-Scarf-squarefreepower} and \ref{lem:2n+1-vertices} imply that $y_1$ and $y_2$ are leaves of $H_1$. Similarly, considering the induced subgraph $H_2$ on $V(G)\setminus \{y_3\}$, we see that $y_1$ and $y_2$ are leaves of $H_2$. Therefore, $y_1$ and $y_2$ are leaves of $G$. By symmetry, it follows that $y_3$ and $y_4$ are leaves of $G$ as well. Having proved our claim, observe that 
        \[
        I(G)^{[n]} = \left(\prod_{i=1}^{2n-2}x_i \right)
        (y_1y_3, y_1y_4, y_2y_3, y_2y_4)
        \]
        which is the product of a monomial and the edge ideal of a $4$-cycle. Then $I(G)^{[n]}$ is not Scarf by Lemma~\ref{lem:multiply-by-a-monomial} and \cite[Theorem 8.3]{FHHM24}, a contradiction. 
\end{proof}

Now we are ready to prove the main theorem of this section.
\begin{theorem}\label{thm:Scarf-squarefree-powers}
Let $G$ be a graph and let $2\leq n\leq m(G)$. Then the following statements are equivalent.
    \begin{enumerate}
        \item $I(G)^{[n]}$ is Scarf.
        \item $I(G)^{[n]}$ is Taylor.
        \item Either $G$ has a perfect matching of size $n$, or the vertex set of $G$ can be labelled as $V(G)=\{x_1,\dots, x_{2n},y_1,\dots y_\ell\}$ for some $\ell\geq 1$ such that \[\{x_1x_2, x_3x_4, \dots, x_{2n-3}x_{2n-2}, x_{2n-1}x_{2n}\}\] is a matching of $G$, and the vertices $x_{2n},y_1,\dots ,y_\ell$ are leaves of $G$ with the common joint $x_{2n-1}$.
        \end{enumerate}
\end{theorem}

\begin{proof}
First, we note that (2) implies (1) by Lemma~\ref{lem:Taylor-implies-Scarf}. If (3) holds, then 
\[I(G)^{[n]}=\left(\prod_{i=1}^{2n-1}x_i\right)(y_1,\dots ,y_\ell,x_{2n}) \]
which implies that every minimal monomial generator of $I(G)^{[n]}$ has a variable which does not divide any other minimal monomial generator. Therefore (3)$\implies$(2). 

    Now, we show that (1)$\implies$(3). Assume that $I(G)^{[n]}$ is Scarf. Then $m(G)=n$ by Lemma~\ref{lem:matching-number-Erey}. Let $\{x_1x_2, x_3x_4,\dots , x_{2n-1}x_{2n}\}$ be a matching of $G$. Let $e_i=x_{2i-1}x_{2i}$. If $G$ has exactly $2n$ vertices, then the result follows. We can then assume that $G$ has at least $2n+1$ vertices, i.e., 
    \[
    V(G)=\{x_1,\dots, x_{2n}, y_1,\dots ,y_\ell\}
    \]
    for some $\ell\geq 1$. If $\ell=1$, then the result follows from Lemmas~\ref{lem:Scarf-2-gens} and \ref{lem:2n+1-vertices}. Hence we assume that $\ell\geq 2$. Since $m(G)=n$, the set of vertices $Y:=\{ y_1,\dots, y_\ell \}$ is independent in $G$. On the other hand, since $G$ has no isolated vertices, every vertex in $Y$ is adjacent to an endpoint of $e_i$ for some $1\leq i\leq n$. Without loss of generality, suppose $y_1$ is adjacent to $x_{2n-1}$. Considering the induced subgraph on $V(G)\setminus \{y_2,\dots ,y_\ell\}$ and applying Lemmas~\ref{lem:HHZ-Scarf-squarefreepower} and \ref{lem:2n+1-vertices}, it follows that $y_1$ and $x_{2n}$ are leaves of $G$.

    Let $2\leq i\leq \ell$. It remains to show that $y_i$ is a leaf of $G$ with joint $x_{2n-1}$. Indeed, $y_ix_{2n}$ is not an edge of $G$ because $m(G)=n$. Moreover, considering the induced subgraph on $\{x_1,\dots,x_{2n}, y_1, y_i\}$ and applying Lemmas~\ref{lem:HHZ-Scarf-squarefreepower} and \ref{lem:forbidden-squarefree}, it follows that $y_i$ is not adjacent to $x_j$ for all $1\leq j\leq 2n-2$, which completes the proof.
\end{proof}

\section{Symbolic powers of edge ideals}\label{sec:symbolic powers}
Let $G$ be a graph with edge ideal $I(G)$ in $S=\Bbbk[x_1,\dots ,x_n]$. For any subset $A$ of $\{x_1,\dots ,x_n\}$, let $\mathfrak{p}_A$ denote the prime ideal generated by the variables in $A$. Let $\mathcal{C}(G)$ denote the set of minimal vertex covers of $G$. It is well-known that $I(G)$ has a unique irredundant primary decomposition (see, e.g., \cite[Corollary 1.3.6]{HH11}) 
\[
I(G)= \bigcap_{A\in \mathcal{C}(G)} \mathfrak{p}_A.
\]
The \emph{$k$-th symbolic power} of $I(G)$, denoted by $I(G)^{(k)}$, is defined to be
\[
I(G)^{(k)} = \bigcap_{A\in \mathcal{C}(G)} \mathfrak{p}_A^k.
\]
It was conjectured by Minh \cite{MV2021} that $\reg(I(G)^n)=\reg(I(G)^{(n)})$ for any graph $G$. Based on this conjecture, one might expect some similarities in cellular resolutions of ordinary and symbolic powers of edge ideals. In fact, we will see that ordinary and symbolic powers of edge ideals are Scarf under very similar conditions. 
We record the following well-known result.

\begin{lemma}[{\cite{SVV94}}]\label{lem:bipartite-symbolic}
    A graph $G$ is bipartite if and only if $I(G)^{(k)}=I(G)^k$ for any $k\geq 1$.
\end{lemma}

We show that symbolic powers interact well with the restriction lemma.

\begin{lemma}\label{lem:MVC-induced}
    Let $G$ be a finite simple graph and let $H$ be an induced subgraph. Then any minimal vertex cover of $G$ contains some minimal vertex cover of $H$, and any minimal vertex cover of $H$ is contained in some minimal vertex cover of $G$.
\end{lemma}
\begin{proof}
    Any minimal vertex cover of $G$ clearly contains a vertex cover of $H$, and thus by definition contains a minimal vertex cover of $H$. This settles the first claim.

    For the second claim, consider a minimal vertex cover $W$ of $H$. Then $W\cup (V(G)\setminus V(H))$ is a vertex cover of $G$, and thus contains a minimal vertex cover of $G$. Since vertices in $V(G)\setminus V(H)$ do not involve any edge in $H$, no vertex of $W$ can be removed from $W\cup (V(G)\setminus V(H))$ in the process of obtaining a minimal vertex cover of $G$. Thus the second claim follows, as desired.
\end{proof}
The above lemma can be rephrased in terms of the minimal primes as follows.
\begin{lemma}\label{lem:MVC ideal version}
    Let $G$ be a finite simple graph and let $H$ be an induced subgraph. Let $m$ be the product of all vertices of $H$. If $\mathfrak{p}$ is a minimal prime of $I(H)$, then there is a minimal prime $\mathfrak{q}$ of $I(G)$ such that $\mathfrak{p}=\mathfrak{q}^{\leq m}$. On the other hand, if $\mathfrak{q}$ is a minimal prime of $I(G)$, then there is a minimal prime $\mathfrak{p}$ of $I(H)$ such that $\mathfrak{p}\subseteq \mathfrak{q}^{\leq m}$.
\end{lemma}
\begin{lemma}\label{lem:HHZ-intersection}
    Let $Q_1,\dots, Q_t$ be monomial ideals where $t$ is an integer. Set $I=\bigcap_{i=1}^t Q_i$. Then for any monomial $m$, we have
    \[
    I^{\leq m} = \bigcap_{i=1}^t Q_i^{\leq m}. 
    \]
\end{lemma}
\begin{proof}
    By induction, it suffices to prove the result for $t=2$. It is clear that $I^{\leq m}\subseteq Q_1^{\leq m} \cap Q_2^{\leq m}$. Conversely, consider a minimal monomial generator $f$ of $Q_1^{\leq m} \cap Q_2^{\leq m}$. It is well-known (see, e.g., \cite[Proposition 1.2.1]{HH11}) that we can write $f=\lcm(a,b)$ where $a\in Q_1^{\leq m} $ and $b\in Q_2^{\leq m}$. Thus $f\in Q_1\cap Q_2=I$ and $f\mid m$. In other words, $f\in I^{\leq m}$, as desired.
\end{proof}

\begin{lemma}\label{lem:squarefree restriction with same support}
    Let $I$ be a squarefree monomial ideal and let $u$ be a squarefree monomial. Then
    \[I^{\leq u}=I^{\leq w}\]
    for any monomial $w$ with $\supp(w)=\supp(u)$.
\end{lemma}
\begin{proof}
    As $I$ is squarefree, $I^{\leq v}=I^{\leq \supp(v)}$ for any monomial $v$ by definition of the restriction.
\end{proof}
\begin{lemma}\label{lem:HHZ-powers} Let $I$ be a squarefree monomial ideal and let $m$ be a squarefree monomial. Then, for any positive integer $n$, 
\[(I^n)^{\leq m^n}=(I^{\leq m})^n.\]
\end{lemma}
\begin{proof}
    Let $u=m^n$. Since $\supp(m)=\supp(u)$, we can apply respectively Proposition~\ref{prop:restriction and power switch}(2) and Lemma~\ref{lem:squarefree restriction with same support} to get 
    \[(I^n)^{\leq u}=\Big[\big(I^{\leq u}\big)^n \Big]^{\leq u}=\Big[\big(I^{\leq m}\big)^n\Big]^{\leq u}.\]
    On the other hand, since the least common multiple of the minimal monomial generators of $(I^{\leq m})^n$ divides $u$, the result follows from Proposition~\ref{prop:restriction and power switch}(1). 
\end{proof}

Gu et al. showed in \cite[Lemma~4.1]{GHOS20} that when $H$ is an induced subgraph of $G$, the symbolic power $I(H)^{(n)}$ is a subideal of $I(G)^{(n)}$. We show that $I(H)^{(n)}$ is in fact an HHZ-subideal of $I(G)^{(n)}$.  The restriction property will then serve as our main tool to show when $I(G)^{(n)}$ is not Scarf.

\begin{theorem}[{Restriction Lemma for the Scarf complexes of symbolic powers}]\label{thm:HHZ-Scarf-symbolicpowers} 
    Let $H$ be an induced subgraph of $G$, and let $n$ be a positive integer. Let $m$ be the product of all vertices of $H$. Then 
    \[\big( I(G)^{(n)} \big)^{\leq m^n} = I(H)^{(n)}.\]
    
    In particular, if $I(G)^{(n)}$ Scarf, then so is $I(H)^{(n)}$.
\end{theorem}
\begin{proof}
    Combining the definition of symbolic powers with Lemma~\ref{lem:HHZ-intersection} and Lemma~\ref{lem:HHZ-powers}  we obtain that
    \[ \left(I(G)^{(n)}\right)^{\leq m^n} 
    = 
    \left(\bigcap_{V\in \mathcal{C}(G)} \mathfrak{p}_V^n\right)^{\leq m^n} 
    = \bigcap_{V\in \mathcal{C}(G)} \left(\mathfrak{p}_V^n\right)^{\leq m^n} 
    = \bigcap_{V\in \mathcal{C}(G)} \left({\mathfrak{p}_V}^{\leq m} \right)^n.\]

    Lemma~\ref{lem:MVC ideal version} implies that
    \[\bigcap_{V\in \mathcal{C}(G)} \big({\mathfrak{p}_V}^{\leq m} \big)^n = \bigcap_{W\in \mathcal{C}(H)} \big({\mathfrak{p}_W} \big)^n =I(H)^{(n)}\]
    as desired. The last statement follows from Lemma~\ref{lem:HHZ-Taylor-Scarf}.
\end{proof}

The graded Betti numbers of $I(G)^{(n)}$ were compared with those of $I(H)^{(n)}$ in \cite{GHOS20} using upper-Koszul simplicial complexes associated to monomial ideals, and Nagata-Zariski’s characterization of symbolic powers of radical ideals over a field of characteristic 0. We recover \cite[Lemma~4.4]{GHOS20} and \cite[Corollary~4.5]{GHOS20} in the next corollary. 
\begin{corollary}Let $H$ be an induced subgraph of $G$ and let $n$ be a positive integer. Then
\begin{enumerate}
    \item $\beta_{i,j}(I(H)^{(n)})\leq \beta_{i,j}(I(G)^{(n)})$,
    \item $\reg I(H)^{(n)} \leq \reg I(G)^{(n)}$.
\end{enumerate}
\end{corollary}
\begin{proof}
    Immediately follows from  \cite[Lemma~4.4]{HHZ04} and Theorem~\ref{thm:HHZ-Scarf-symbolicpowers}.
\end{proof}

Next, we note down a known class of Scarf monomial ideals from \cite{MSY00}. A monomial ideal $I$ is called \emph{generic} if for any two distinct minimal monomial generators $m,m'$ of $I$ with the same positive degree in some variable, there exists a third minimal monomial generator $m''$ such that $m''\mid \lcm(m,m')$ and $\supp(\lcm(m,m'))=\supp(\lcm(m,m')/m'')$. We record the following result.

\begin{lemma}[{\cite[Theorem 1.5]{MSY00}}]\label{lem:generic-ideals}
    Generic monomial ideals are Scarf.
\end{lemma}

Now that we have the main tools ready, we proceed to the main result of the section. The goal of this section is the complete characterization of pairs $(G,n)$ such that the symbolic power $I(G)^{(n)}$ is Scarf. As with squarefree powers, $I(G)^{(1)}=I(G)$, and thus we already have the characterization in this case \cite[Theorem 8.3]{FHHM24}. For the remainder of the section, we assume that $n\geq 2$.

First, we consider $G$ to be a triangle and characterize when $I(G)^{(n)}$ is Scarf. 


\begin{proposition}\label{prop:symbolic-triangle}
    Let $G$ be a triangle graph. Then $I(G)^{(n)}$ is Scarf if and only if $n$ is even.
\end{proposition}

\begin{proof}
    Let $I(G)=(xy,xz,yz)=(x,y)\cap (x,z)\cap (y,z)$. We will deal with the cases when $n$ is odd or even separately.

    Suppose that $n$ is odd. Due to the discussion before the proposition, we can assume that $n\geq 3$. Set $m\coloneqq (xyz)^{(n+1)/2}$. Then we have
    \begin{align*}
        (I(G)^{(n)})^{\leq m} &= \big(  (x,y)^n \big)^{\leq m} \cap\big(  (x,z)^n \big)^{\leq m} \cap\big(  (y,z)^n \big)^{\leq m} \\
        &= \big(  x^{\frac{n-1}{2}}y^{\frac{n+1}{2}}, x^{\frac{n+1}{2}}y^{\frac{n-1}{2}} \big) \cap\big(  x^{\frac{n-1}{2}}z^{\frac{n+1}{2}}, x^{\frac{n+1}{2}}z^{\frac{n-1}{2}} \big)\cap\big( y^{\frac{n-1}{2}}z^{\frac{n+1}{2}}, y^{\frac{n+1}{2}}z^{\frac{n-1}{2}}  \big) \\
        &= (xyz)^{\frac{n-1}{2}} I(G).
    \end{align*}
    Since $I(G)$ is not Scarf, neither is $I(G)^n$ by Lemma~\ref{lem:multiply-by-a-monomial}, as desired.

    Now suppose that $n$ is even. To show that $I(G)^{(n)}$ is Scarf, by Lemma~\ref{lem:generic-ideals}, it suffices to show that $I$ is generic. To do so, we need an explicit description of the minimal monomial generators, which is given in the following claim. 
    \begin{claim}\label{clm:generators-C3}
        We have
        \[
        I(G)^{(n)}=(x,y)^n \cap  (x,z)^n \cap  (y,z)^n = (x^{n-i}y^{n-i}z^i, x^{n-i}y^{i}z^{n-i}, x^{i}y^{n-i}z^{n-i}\colon 0\leq i \leq n/2 ).
        \]
    \end{claim}
    Assume the claim. Given two distinct monomial generators $m,m'$ of $I(G)^{(n)}$ with the same positive degree in some variable. Then without loss of generality, we can assume $m=x^{n-i}y^{n-i}z^i$ and $m'=x^{n-i}y^{i}z^{n-i}$ for some $i\leq  n/2$, and since $m\neq m'$, we have $i<n/2$, or equivalently, $n-i>n/2$. Then we can set $m''= (xyz)^{n/2}$. We observe that $\lcm(m,m')=(xyz)^{n/2}$, and thus $m''\mid \lcm(m,m')$, and $\lcm(m,m')=(xyz)^{n/2}$ and $\lcm(m,m')/m''= (xyz)^{n-i-n/2}$ have the same support, as desired.
\end{proof}

\begin{proof}[Proof of Claim~\ref{clm:generators-C3}]
    The first equality is from our definition of symbolic powers. For the second equality, the reverse inclusion $\supseteq$ is clear. For the inclusion $\subseteq$, without loss of generality, it suffices to show that if $x^ay^bz^c$, where $a\geq b\geq c$, is minimal monomial generator of $I(G)^{(n)}$, then $a=b=n-c$, and $c\leq n/2$. 

    First we note that a monomial $x^{l_1}y^{l_2}z^{l_3}$ is in $I(G)^{(n)}=(x,y)^n \cap  (x,z)^n \cap  (y,z)^n $ if and only if we have three inequalities
    \[
    l_1+l_2 \geq n, \quad l_1+l_3 \geq n, \quad l_2+l_3 \geq n.
    \]
    Since $x^ay^bz^c$ is a minimal generator of $I(G)^{(n)}$, we have $b+c\geq n$, and the monomial $x^{a-1}y^bz^c$ is not in $I(G)^{(n)}$. Hence we have either $a-1+b<n$ or $a-1+c<n$. Since the former implies the latter, we always have $a-1+c<n$. Thus $b-1+c\leq a-1+c<n$, or $b+c<n+1$. This forces $b+c=n$. In particular, this means $a-1+c<n = b+c$, or equivalently, $a<b+1$. This forces $a=b$. Finally, we have $n\geq b+c \geq 2c$, and thus $c\leq n/2$, as desired.
\end{proof}


Next, we will provide a list of forbidden structures for when $I(G)^{(n)}$ is Scarf.

\begin{proposition}\label{prop:forbidden-symbolic}
    Assume that $n\geq 2$. Let $G$ be one of following graphs:
    \begin{enumerate}
        \item a claw graph \clawsymb, a path of length $3$, or a $4$-cycle $C_4$;
        \item a paw graph \pawsymb;
        \item a diamond graph \diamondsymb;
        \item a complete graph on 4 vertices \Kfoursymb;
        \item the disjoint union of an edge and a path of length 2;
        \item the disjoint union of an edge and a triangle;  
        \item the disjoint union of three edges.
    \end{enumerate}
    Then $I(G)^{(n)}$ is not Scarf.
\end{proposition}

\begin{proof}
    Except for the first case, we will find an HHZ-subideal of $I(G)^{(n)}$ that is not Scarf, and thus the results would follow from Lemma~\ref{lem:HHZ-Taylor-Scarf}.
    \begin{enumerate}
        \item If $G$ is a claw graph \clawsymb, a path of length $3$, or a $4$-cycle, then $G$ is bipartite, and thus $I(G)^{(n)} = I(G)^n$ by Lemma~\ref{lem:bipartite-symbolic}. Thus $I(G)^{(n)}$ is not Scarf by \cite[Theorem 8.3]{FHHM24}.
        \item If $G$ is a paw graph \pawsymb, i.e., $I(G)=(wx,xy,yz,xz)$, then we have
        \[
        I(G) =  (w,y,z) \cap (x,y) \cap (x,z).
        \]
        Setting $m\coloneqq w^{n-1}x^nyz$ and keeping Proposition~\ref{prop:restriction and power switch} and Lemma~\ref{lem:HHZ-intersection} in mind, we get
        \begin{align*}
            (I(G)^{(n)})^{\leq m}&= \big((w,y,z)^n\big)^{\leq m} \cap \big((x,y)^n\big)^{\leq m}\cap \big((x,z)^n\big)^{\leq m}\\
            &=\big( w^{n-1}y,w^{n-1}z, w^{n-2}yz \big) \cap \big( x^n,x^{n-1}y \big) \cap \big( x^n,x^{n-1}z \big)  \\
            &=w^{n-2}x^{n-1}(yz,wxy,wxz).
        \end{align*}  
        Observe that the Scarf simplicial complex associated to the ideal $(yz,wxy,wxz)$ is disconnected. Therefore $(yz,wxy,wxz)$ is not Scarf by Lemma~\ref{lem:BPS}, and thus neither is $(I(G)^{(n)})^{\leq m}$ by Lemma~\ref{lem:multiply-by-a-monomial}, as desired.

        
        \item If $G$ is a diamond graph, i.e., $I(G)=(wx,wy,wz,xy,yz)$, then we have
        \[
        I(G) =  (w,y) \cap (w,x,z) \cap (x,y,z).
        \]
        Set $m\coloneqq w^{n-1}xy^{n-1}z$. Then 
        \begin{align*}
            (I(G)^{(n)})^{\leq m}&= \big((w,y)^n\big)^{\leq m} \cap \big((w,x,z)^n\big)^{\leq m}\cap \big((x,y,z)^n\big)^{\leq m}\\
            &=\big( (w,y)^n \big)^{\leq m} \cap \big( w^{n-1}x, w^{n-1}z, w^{n-2}xz \big) \cap \big( xy^{n-1}, y^{n-1}z, xy^{n-2}z  \big)  \\
            &=\big( w^{n-1}y, w^{n-2}y^2,\dots, wy^{n-1} \big) \cap \big( w^{n-1}xy^{n-1},w^{n-1}y^{n-1}z,w^{n-2}xy^{n-2}z  \big)  \\
            &=w^{n-2}y^{n-2}(wxy,wxz,wyz,xyz).
        \end{align*}
        By Lemma~\ref{lem:Scarf-2-gens}, the ideal $(wxy,wxz,wyz,xyz)$ is not Scarf. Then by Lemma~\ref{lem:multiply-by-a-monomial} $(I(G)^{(n)})^{\leq m}$ is not Scarf either, as desired.

        
        \item If $G$ is a complete graph on 4 vertices, i.e., $I(G)=(wx,wy,wz,xy,xz,yz)$,  then we first observe that $I(G)^n$ is not Scarf if $n$ is odd by Proposition~\ref{prop:symbolic-triangle} since $G$ contains an induced triangle. Thus we can assume that $n=2k$ for some $k\geq 1$. We have the following decomposition:
        \[
        I(G) =  (w,x,y) \cap (w,x,z) \cap (w,y,z) \cap (x,y,z).
        \]
        Set $m\coloneqq wx^ky^{k}z^{k}$. Then
        \begin{align*}
            (I(G)^{(n)})^{\leq m}&= \big((w,x,y)^n\big)^{\leq m} \cap \big((w,x,z)^n\big)^{\leq m}\cap \big((w,y,z)^n\big)^{\leq m}\cap \big((x,y,z)^n\big)^{\leq m}\\
            &=\begin{multlined}[t]
                \big( wx^{k}y^{k-1}, wx^{k-1}y^k,x^ky^k  \big) \cap \big(wx^{k}z^{k-1}, wx^{k-1}z^k,x^kz^k  \big) \cap\\
                \big( wy^{k}z^{k-1}, wy^{k-1}z^k,y^kz^k  \big) \cap \big( (x,y,z)^n  \big)^{\leq m}
            \end{multlined}  \\
            &= \big( x^{k-1}y^{k-1}z^{k-1} (wxy,wxz,wyz,xyz) \big)\cap \big( (x,y,z)^n  \big)^{\leq m} \\
            &=x^{k-1}y^{k-1}z^{k-1} \big( wxy,wxz,wyz,xyz \big).
        \end{align*}
        By  Lemma~\ref{lem:Scarf-2-gens}, the ideal $(wxy,wxz,wyz,xyz)$ is not Scarf, and hence neither is $(I(G)^{(n)})^{\leq m}$ by Lemma~\ref{lem:multiply-by-a-monomial}, as desired.
        
        \item If $G$ is the disjoint union of an edge and a path of length $2$, i.e., $I(G)=(uv, xy, yz)$, then $G$ is bipartite and thus $I(G)^{(n)}=I(G)^n$ by Lemma~\ref{lem:bipartite-symbolic}.
        Set $m\coloneqq u^{n-1}v^{n-1}xy^{2}z$. Then  we have
        \begin{align*}
            (I(G)^{(n)})^{\leq m} = (I(G)^{n})^{\leq m}=  u^{n-2}v^{n-2}(uvxy,uvyz,xy^2z).
        \end{align*}
        As in case (2), one can see that $(uvxy,uvyz,xy^2z)$ is not Scarf. By Lemma~\ref{lem:multiply-by-a-monomial}, it follows that $(I(G)^{(n)})^{\leq m}$ is not Scarf either, as desired.


         \item Let $G$ be the disjoint union of an edge and a triangle, i.e., $I(G)=(uv, xy, xz, yz)$. Then,
         \[I(G)=(u,x,y)\cap (u,x,z)\cap (u,y,z)\cap (v,x,y)\cap (v,x,z)\cap (v,y,z). \]
        Set $m\coloneqq u^{n-1}v^{n-1}xyz$. Observe that $(I(G)^{(n)})^{\leq m}=I_1\cap I_2$ for
        \[I_1= \big((u,x,y)^n\big)^{\leq m} \cap \big((u,x,z)^n\big)^{\leq m} \cap \big((u,y,y)^n\big)^{\leq m} \]
        and 
        \[I_2=\big((v,x,y)^n\big)^{\leq m} \cap \big((v,x,z)^n\big)^{\leq m} \cap \big((v,y,z)^n\big)^{\leq m}.\]
        Now, we obtain
        \begin{align*}
           I_1 & = \big( u^{n-1}x, u^{n-1}y , u^{n-2}xy\big) \cap \big(u^{n-1}x, u^{n-1}z, u^{n-2}xz\big) \cap \big(u^{n-1}y, u^{n-1}z, u^{n-2}yz\big) \\
            & =  u^{n-2} (uxy,uxz,uyz,xyz).
        \end{align*}
        By the symmetry, $I_2= v^{n-2} (vxy,vxz,vyz,xyz)$. Therefore,
        \[(I(G)^{(n)})^{\leq m}=I_1\cap I_2=u^{n-2}v^{n-2} (uvxy,uvxz,uvyz,xyz).\]
        The end of the proof is the same as Case (5).
        
        \item If $G$ is the disjoint union of three edges, i.e., $I(G)=(uv, wx, yz)$, then $G$ is bipartite and
        thus $I(G)^{(n)}=I(G)^n$ by Lemma~\ref{lem:bipartite-symbolic}.
        Setting $m\coloneqq u^{n-1}v^{n-1}wxyz$, we see that
        \begin{align*}
            (I(G)^{(n)})^{\leq m} = (I(G)^{n})^{\leq m}=  u^{n-2}v^{n-2}(uvwx,uvyz,wxyz).
        \end{align*}
         The proof follows as in Case (2).\qedhere
         \end{enumerate} 
         \end{proof}

We are now ready to prove the main theorem of this section.

\begin{theorem}\label{thm:Scarf-symbolic-powers}
    Assume $n\geq 2$. Let $G$ be a graph without isolated vertices. Then $I(G)^{(n)}$ is Scarf if and only if one of the following holds:
    \begin{enumerate}
        \item $G$ is a path of length 1 or 2;
        \item $G$ is a triangle and $n$ is even;
        \item $G$ is the disjoint union of two edges.
    \end{enumerate}
\end{theorem}

\begin{proof}
    Suppose that one of the above scenarios occurs. If $G$ is a path of length $1$ or $2$, the result follows from \cite[Theorem 8.3]{FHHM24}. If $G$ is a triangle and $n$ is even, the result follows from Proposition~\ref{prop:symbolic-triangle}. Finally, assume that $G$ is the disjoint union of two edges, i.e., $I(G)=(wx,yz)$, then $G$ is bipartite, and thus $I(G)^{(n)}=I(G)^n$. It is straightforward that $I(G)^{(n)}$ is generic since no two distinct minimal monomial generators share the same positive degree in any variable. By Lemma~\ref{lem:generic-ideals}, $I(G)^{(n)}$ is Scarf, as desired.
    
    Conversely, suppose that $I(G)^{(n)}$ is Scarf. Due to Theorem~\ref{thm:HHZ-Scarf-symbolicpowers}, $G$ does not contain any graph in the list in Proposition~\ref{prop:forbidden-symbolic} as an induced subgraph. Since $G$ does not contain the disjoint union of three edges, it has at most two connected components. If $G$ has two connected components, then $G$ must be the disjoint union of two edges because of items (5) and (6) of Proposition~\ref{prop:forbidden-symbolic}.
    
    Now, we can assume $G$ is connected. Since $G$ has no induced claw \clawsymb, paw \pawsymb, diamond \diamondsymb, or $K_4$ \Kfoursymb, it follows that every vertex of $G$ has degree at most $2$. This implies that $G$ is either a path or a cycle. If $G$ is a cycle, then by (1) of Proposition~\ref{prop:forbidden-symbolic}, $G$ must be a triangle and $n$ must be even by Proposition~\ref{prop:symbolic-triangle}. If $G$ is a path, then it must be of length $1$ or $2$ by the first item of Proposition~\ref{prop:forbidden-symbolic}, as desired.
\end{proof}

We can also extend the main theorem in \cite{FHHM24} from connected graphs to all graphs. First, we note down the Restriction Lemma for ordinary powers of edge ideals.

\begin{lemma}[{Restriction Lemma for the Scarf complexes of ordinary powers}]\label{lem:HHZ-Scarf-regularpowers}  
    Let $H$ be an induced subgraph of $G$, and let $n$ be a positive integer. Let $m$ be the product of all vertices of $H$. Then 
    \[\big( I(G)^{n} \big)^{\leq m^n} = I(H)^{n}.\]
    
    In particular, if $I(G)^{n}$ is Scarf, then so is $I(H)^{n}$.
\end{lemma}
\begin{proof}
    The result follows immediately from Lemma~\ref{lem:HHZ-Taylor-Scarf} and Lemma~\ref{lem:HHZ-powers}.
\end{proof}

Now, we are ready to prove our result on Scarf ordinary powers of edge ideals.

\begin{theorem}\label{thm:Scarf-regular-powers}
    Assume $n\geq 2$. Let $G$ be a graph without isolated vertices. Then $I(G)^{n}$ is Scarf if and only if one of the following holds:
    \begin{enumerate}
        \item $G$ is a path of length 1 or 2;
        \item $G$ is the disjoint union of two edges.
    \end{enumerate}
\end{theorem}
\begin{proof}
    The result follows from \cite[Theorem 8.3]{FHHM24} if $G$ is connected. So, now we can assume that $G$ is not connected. 

    Suppose that one of the above scenarios occurs. Then $G$ is bipartite, and thus $I(G)^n=I(G)^{(n)}$ is Scarf by Lemma~\ref{lem:bipartite-symbolic} and Theorem~\ref{thm:Scarf-symbolic-powers}.

    Conversely, suppose that $I(G)^n$ is Scarf. Due to Lemma~\ref{lem:bipartite-symbolic}, Theorem~\ref{thm:Scarf-symbolic-powers} and Lemma~\ref{lem:HHZ-Scarf-regularpowers}, $G$ does not contain the disjoint union of three edges as an induced subgraph. Therefore, $G$ has exactly two connected components. Because of Lemma~\ref{lem:HHZ-Scarf-regularpowers}
    and \cite[Theorem 8.3]{FHHM24}, each connected component of $G$ is either an edge or a path of length $2$. The result then follows from Lemma~\ref{lem:bipartite-symbolic} and Theorem~\ref{thm:Scarf-symbolic-powers}.
\end{proof}

\section{Cover ideals of graphs}\label{sec:cover ideals}

The goal of this section is to characterize graphs $G$ such that $J(G)$ is Scarf.

For any vertex $v$ of $G$, let $N(v)$, called the \emph{open neighborhood} of $v$, denote the set of vertices adjacent to $v$, and $N[v]$, called the \emph{closed neighborhood} of $v$, denote the set $N(v)\cup \{v\}$. We first study when two generators of $J(G)$ form an edge in its Scarf complex.

\begin{lemma}\label{lem:Scarf edge for cover ideals}
    Let $\{m_1, m_2\}$ be an edge of the Scarf complex of the cover ideal $J(G)$. Let $x$ and $y$ be two vertices such that $x\in \supp(m_1)\setminus \supp(m_2)$ and $y\in \supp(m_2)\setminus \supp(m_1)$. Then $xy\in I(G)$.
\end{lemma}
\begin{proof}
    Assume for a contradiction $\{x,y\}$ is not an edge of $G$. Let $W=N[x]\cup N[y]$. Let $Z=\supp(m_1)\setminus W$. Observe that $C:=Z\cup N(x)\cup N(y)$ is a vertex cover of $G$ because $\{x,y\}$ is not an edge of $G$. Let $m'$ be the monomial corresponding to $C$. Note that neither $x$ nor $y$ divides $m'$. Since every vertex cover can be reduced to a minimal one, there exists a minimal generator $m_3$ of $J(G)$ which divides $m'$. Now, since $x\notin \supp(m_2)$ we have $N(x)\subseteq \supp(m_2)$. Similarly, $N(y)\subseteq \supp(m_1)$. It follows that $m_3$ divides $\lcm(m_1,m_2)$, which is a contradiction.
\end{proof}

It turns out that the Scarf complex of a cover ideal is of low dimension.

\begin{lemma}\label{lem:scarf-dim-1}
    The Scarf complex of a cover ideal is of dimension at most $1$.
\end{lemma}
\begin{proof}
    Assume for a contradiction $\{m_1, m_2, m_3\}$ is a $2$-dimensional face of the Scarf complex of $J(G)$. Since $m_i$ does not divide $\lcm(m_j,m_k)$ for all distinct $i,j$ and $k$, it follows that there exist variables $x_1$, $x_2$ and $x_3$ such that $x_1\in \supp(m_1)\setminus \supp(m_2m_3)$, $x_2\in \supp(m_2)\setminus \supp(m_1m_3)$ and $x_3\in \supp(m_3)\setminus \supp(m_1m_2)$. Lemma~\ref{lem:Scarf edge for cover ideals} implies that $\{x_1,x_2\}$ is an edge of $G$. On the other hand, neither $x_1$ nor $x_2$ divides $m_3$. This contradicts the fact that the support of $m_3$ is a vertex cover of $G$.
    \end{proof}

Due to a well-known theorem of Fröberg \cite{Fro90}, $\reg(S/I(G))=1$ if and only if $G$ is co-chordal. Also, by a well-known theorem of Terai \cite{Terai99}, we have $\pd (J(G))=\reg(S/I(G))$ for any graph $G$. Therefore, the above lemma implies that if a cover ideal of a graph is Scarf, then the graph must be co-chordal. The above lemma thus allows us to convert the algebraic property of cover ideal being Scarf into a purely graph-theoretic property. We therefore assume that $G$ is co-chordal for the rest of the section. 

We now state a main theorem of this section, fully characterizing Scarf cover ideals using the existence of some of the edges of its Scarf complex. For an ideal $I$, we denote by $\mu(I)$, the minimum number of generators of $I$.

\begin{theorem}\label{thm:Scarf-cover-ideals}
    Let $G$ be a graph. Then the following are equivalent.
    \begin{enumerate}
        \item $J(G)$ is Scarf.
        \item $G$ is co-chordal and the Scarf complex of $J(G)$ is connected.
        \item $G$ is co-chordal and there are exactly $\mu(J(G))-1$ pairs of minimal vertex covers $(V_1,V_2)$ such that  the only minimal vertex covers of $G$ within $V_1\cup V_2$ are $V_1$ and $V_2$ themselves.
        \item $G$ is co-chordal and there are at least $\mu(J(G))-1$ pairs of minimal vertex covers $(V_1,V_2)$ such that  the only minimal vertex covers of $G$ within $V_1\cup V_2$ are $V_1$ and $V_2$ themselves.
    \end{enumerate}
\end{theorem}
\begin{proof}
    Set $\mu(J(G))=n$. For any set $W$ of vertices of $G$, let $m_W$ denote the monomial $\prod_{u\in W}u$. We first show that a pair of minimal vertex covers $(V_1,V_2)$ satisfies the property that the only minimal vertex covers of $G$ within $V_1\cup V_2$ are $V_1$ and $V_2$ themselves if and only if $\{m_{V_1},m_{V_2}\}$ is Scarf. Indeed, $\{m_{V_1},m_{V_2}\}$ is Scarf if and only if there exists no minimal vertex cover $V\notin \{V_1,V_2\}$ such that $m_V\mid \lcm(m_{V_1},m_{V_2})$, which is equivalent to the above property, as desired.
    
    To prove that $(1)\implies (2)$, assume that $J(G)$ is Scarf. Then $\pd(J(G))=1$ by the preceding result. Since the Scarf complex is a subcomplex of the minimal resolution, it is also of dimension 1, at most. By Lemma~\ref{lem:BPS}, this Scarf complex is a tree, which is (2). 
    
    $(2)\implies (3)$ follows similarly since a tree on $n$ vertices has exactly $n-1$ edges.

    It is clear that $(3)\implies (4)$. Finally, we show $(4)\implies (1)$. Assume that $G$ is co-chordal and there are at least $\mu(J(G))-1$ pairs of minimal vertex covers $(V_1,V_2)$ such that  the only minimal vertex covers of $G$ within $V_1\cup V_2$ are $V_1$ and $V_2$ themselves. Since $G$ is co-chordal, $\pd(J(G)) = 1$ by Terai's and Fr\"oberg's theorems. Thus the minimal resolution of $J(G)$ is of the form
    \[
    Q^{b} \to Q^a.
    \]
    It is well-known that $a=\mu(J(G)) = n$. By \cite[Corollary 1.4.6]{BH98}, $b-a+1=0$, or $b=n-1$. Thus the minimal resolution of $J(G)$ is 
    \[
    Q^{n-1} \to Q^n.
    \]
    By the second assumption, the Scarf complex of $J(G)$ contains the following complex
    \[
    Q^{n-1} \to Q^n
    \]
    and therefore is at least as big as the minimal resolution. On the other hand, the minimal resolution always contains the Scarf complex. Therefore the two coincide, and in particular $J(G)$ is Scarf, as desired.
\end{proof} 

Unlike ordinary and symbolic powers of edge ideals, there is a large class of graphs whose cover ideal is Scarf. In Theorem~\ref{thm:Scarf-cover-ideals} we gave a combinatorial description of when $J(G)$ is Scarf. In the rest of this section, we will describe Scarfness of $J(G)$ more explicitly when $G$ has some special properties. We remark that $G$ needs to be co-chordal for $J(G)$ to be Scarf. In particular, $G$ does not contain $C_k$ as an induced subgraph for any $k\geq 5$. Thus an induced cycle of $G$, if exists, must be either a triangle or a $4$-cycle. We will investigate the two extremes: when $G$ is chordal (i.e., all induced cycles of $G$, if any, are triangles), or $G$ is a bipartite graph (i.e., all induced cycles of $G$, if exist, are $4$-cycles).

\begin{theorem}\label{thm:Scarf-cover-ideal-chordal}
    Let $G$ be a chordal graph. Then $J(G)$ is Scarf if and only if $V(G)$ can be partitioned into a clique $A$ and an independent set $B$ such that
    \begin{enumerate}
        \item any vertex in $A$ is adjacent to a vertex in $B$;
        \item no two sets in the collection $\{ N(x) \cap B \colon x\in A \}$ contain one another. 
    \end{enumerate}
\end{theorem}

\begin{proof}
    First, we will show that if $G$ is chordal and co-chordal, then $V(G)$ can be partitioned into a clique $A$ and an independent set $B$ where condition (1) holds. Indeed, by a theorem of Foldes and Hammer \cite{FH77}, $V(G)$ can be partitioned into a clique $A$ and an independent set $B$. Now suppose that (1) does not hold, i.e., there exists a vertex in $A$ that is not adjacent to any vertex in $B$. Then we replace $A$ with $A\setminus \{x\}$ and $B$ with $B\cup \{x\}$ where $x$ is any vertex in $A$. We note that under this change, $A$ is still a clique and $B$ is still an independent set. In other words, such a partition always exists.

    Back to the main problem, we first assume that $J(G)$ is Scarf. Then $G$ is co-chordal, and since $G$ is already chordal,  there exists a partition of $V(G)$ into $A$ and $B$ that satisfies condition (1). Set $A=\{x_1,\dots, x_n\}$. Then by condition (1), it is straightforward that
    \[
    J(G)= (x_1\cdots x_n, \quad \widehat{x_1} \cdots x_n(\prod_{y\in N(x_1)\cap B} y),\quad \dots, \quad x_1\cdots \widehat{x_n} (\prod_{y\in N(x_n)\cap B} y)).
    \]
    To make the notations easy, set $m_i= x_1\cdots \widehat{x_i}\cdots x_n (\prod_{y\in N(x_i)\cap B} y)$. Then, \[J(G)=(x_1\cdots x_n, m_1,\dots , m_n).\]
    We will count the number of possible Scarf edges of $J(G)$. Since $x_1\cdots x_n \mid \lcm(m_i,m_j)$ for any $i\neq j$, the edge $\{m_i,m_j\}$ is not Scarf. Thus there are at most $n= \mu(J(G))-1$ possible Scarf edges for the ideal $J(G)$. By Theorem~\ref{thm:Scarf-cover-ideals}, $J(G)$ is Scarf if and only if these $n$ edges are indeed Scarf. In other words, we must have
    \[
    m_i \nmid \lcm(x_1\cdots x_n, m_j)
    \]
    for any $i\neq j$. By writing down the monomials explicitly, this is the same as
    \[
    \prod_{y\in N(x_i)\cap B} y \nmid \prod_{y\in N(x_j)\cap B} y
    \]
    for any $i\neq j$. This is equivalent to condition (2), as desired.

    Conversely, assume that $V(G)$ can be partitioned into a clique $A$ and an independent set $B$ that satisfy the two conditions $(1)$ and $(2)$. Then it is clear that $J(G)$ is of the form above and again, by Theorem~\ref{thm:Scarf-cover-ideals}, is Scarf.
\end{proof}

Next, we will fully characterize bipartite graphs whose cover ideal is Scarf. A co-chordal bipartite graph is also known as a Ferrers graph (see, e.g., \cite[Theorem 4.2]{CN08}). We recall the definition in \cite{CN08}: a bipartite graph $G$ on the two disjoint sets of vertices $X=\{x_1,\dots, x_n\}$ and $Y=\{y_1,\dots, y_m\}$ of $G$ is called a \emph{Ferrers graph} if up to relabelling, we have
    \[
    N(x_1)\subseteq N(x_2) \subseteq \cdots \subseteq N(x_n)=Y.
    \] 
Since neighbors of $x_i$ can only be in $Y$, if $G$ is a Ferrers graph, we also have
\[
X= N(y_1)\supseteq N(y_2) \supseteq \cdots \supseteq N(y_m).
\]
We will explicitly describe all the minimal monomial generators of $J(G)$ for a Ferrers graph $G$.

\begin{lemma}\label{lem:Ferrers-MVC-1}
    Let $G$ be a Ferrers graph on the two sets of vertices $X=\{x_1,\dots, x_n\}$ and $Y=\{y_1,\dots, y_m\}$ where
    \[
    N(x_1)\subseteq N(x_2) \subseteq \cdots \subseteq N(x_n)=Y.
    \]
    Then the collection of minimal elements of 
    \[
    \Omega\coloneqq \{ N(y_1) = X,\quad \{y_1\} \cup N(y_2),\quad \{y_1,y_2\} \cup N(y_3), \quad \dots,\quad \{y_1,\dots, y_m\} =Y \}
    \]
    is exactly the collection of all minimal vertex covers of $G$.
\end{lemma}
\begin{proof}
    It suffices to show that any minimal vertex cover $C$ of $G$ contains one of these sets, since the minimal elements of $\Omega$ do not contain one another by design. This can be done inductively. Indeed, if $y_1\notin C$, then $N(y_1)\subseteq C$ since $C$ is a vertex cover, as desired. Suppose by induction that $y_1,\dots y_t\in C$, and $y_{t+1}\notin C$. Then $N(y_{t+1})\subseteq C$ since $C$ is a vertex cover. Thus $\{y_1,\dots, y_t\}\cup N(y_{t+1})\subseteq C$, as desired. This holds for all $t\leq m-1$. The last case is when $y_1,\dots, y_m\in C$, or equivalently, $Y\subseteq C$, as desired. This concludes the proof.
\end{proof}

\begin{lemma}\label{lem:Ferrers-MVC-2}
    Let $G$ be a Ferrers graph on the two sets of vertices $X=\{x_1,\dots, x_n\}$ and $Y=\{y_1,\dots, y_m\}$ where
    \[
    N(x_1)\subseteq N(x_2) \subseteq \cdots \subseteq N(x_n)=Y.
    \]
    Let $1< k_1 < \cdots < k_l\leq m$ be indices such that
    \begin{multline*}
        N(y_1)= \cdots = N(y_{k_1-1})  \supsetneq N(y_{k_1}) = \cdots = N(y_{k_2-1})  \supsetneq\\
        N(y_{k_2})  \cdots  N(y_{k_l-1}) \supsetneq N(y_{k_l}) =\cdots = N(y_m),
    \end{multline*}
    i.e., these are exactly the indices where we have strict inclusions. Then
    \[
    \mathcal{M}\coloneqq \{ N(y_1) = X, \quad \{y_1,\dots, y_{k_1-1}\} \cup N(y_{k_1}),\quad \dots, \quad \{y_1,\dots, y_{k_l-1}\} \cup N(y_{k_l}),\quad Y  \}
    \]
    is exactly the collection of all minimal vertex covers of $G$.
\end{lemma}

\begin{proof}
    By Lemma~\ref{lem:Ferrers-MVC-1}, it suffices to show that these are exactly the minimal elements of $\Omega$, where $\Omega$ is defined as in Lemma~\ref{lem:Ferrers-MVC-1}. Indeed, if $N(y_j)=N(y_{j+1})$, then $\{y_1,\dots y_{j}\}\cup N(y_{j+1})$ is by definition not a minimal element of $\Omega$. Thus the minimal elements of $\Omega$ is indeed a subset of~$\mathcal{M}$.

    Conversely, we will show that any element in $\mathcal{M}$ is minimal. Indeed, both $X$ and $Y$  straightforwardly do not contain, or are contained in, any of the rest, or each other. Thus it suffices to show that elements in 
    \[
     \{ \{y_1,\dots, y_{k_1-1}\} \cup N(y_{k_1}),\quad  \dots, \quad \{y_1,\dots, y_{k_l-1}\} \cup N(y_{k_l})  \}
    \]
    do not contain each other. This is clear from the definition of our indices.
\end{proof}

We are now ready to fully characterize bipartite graphs whose cover ideal is Scarf. 

\begin{theorem}\label{thm:Scarf-cover-ideals-bipartite}
    Let $G$ be a bipartite graph. Then the following statements are equivalent.
    \begin{enumerate}
        \item $J(G)$ is Scarf.
        \item $G$ is co-chordal.
        \item $G$ is a Ferrers graph.
    \end{enumerate}
\end{theorem}
\begin{proof}
    This follows directly from Lemma \ref{lem:Ferrers-MVC-2} and Theorem \ref{thm:Scarf-cover-ideals}.
\end{proof}
We end the paper with a related natural problem:
\begin{problem}
    Given a positive integer $n\geq 2$, classify all graphs $G$ such that $J(G)^n$ is Scarf.
\end{problem}

\bibliographystyle{amsplain}
\bibliography{refs}

\end{document}